\documentclass{amsart}
\usepackage{a4wide,hyperref}
\usepackage{mathrsfs}
\usepackage{mathtools}
\usepackage{tikz-cd}
\usepackage{amsmath}
\usepackage{amssymb}
\numberwithin{equation}{section}
\usepackage[latin1]{inputenc}

\newcommand{\N}{\mathbb{N}}
\newcommand{\Q}{\mathbb{Q}}
\newcommand{\R}{\mathbb{R}}

\newcommand{\X}{{\rm X}}
\newcommand{\Y}{{\rm Y}}
\newcommand{\sfd}{{\sf d}}

\renewcommand{\d}{{\mathrm d}}

\newcommand{\nchi}{{\raise.3ex\hbox{$\chi$}}}

\newcommand{\restr}[1]{\lower3pt\hbox{$|_{#1}$}}

\newcommand{\lims}{\varlimsup}
\newcommand{\fr}{\penalty-20\null\hfill$\blacksquare$}

\newcommand{\mm}{\mathfrak m}

\newcommand{\ud}{\underline{\mathrm d}}
\newcommand{\Lip}{{\rm Lip}}
\newcommand{\lip}{{\rm lip}}

\DeclareMathOperator*{\esssup}{ess\,sup}

\newcommand{\F}{\mathcal F}

\newcommand{\LIP}{\operatorname{LIP}}
\newcommand{\Sob}{{\rm S}}
\newtheorem{theorem}{Theorem}[section]

\newtheorem{lemma}[theorem]{Lemma}
\newtheorem{proposition}[theorem]{Proposition}
\newtheorem{definition}[theorem]{Definition}

\newtheorem{remark}[theorem]{Remark}
\title{Differential of metric valued Sobolev maps}
\thanks{Thanks to MIUR SIR-grant `Nonsmooth Differential Geometry' (RBSI147UG4)}
\author{Nicola Gigli, Enrico Pasqualetto, Elefterios Soultanis}
\date{\today}

\address{SISSA, Via Bonomea 265, 34136 Trieste}
\email{ngigli@sissa.it}
\email{epasqual@sissa.it}
\email{elefterios.soultanis@gmail.com}

\begin{document}

\begin{abstract}
We introduce a notion of differential of a Sobolev map between metric spaces. The differential is given in the framework of tangent and cotangent modules of metric measure spaces, developed by the first author. We prove that our notion is consistent with Kirchheim's metric differential when the source is a Euclidean space, and with the  abstract differential provided by the first author when the target is $\R$. 
\end{abstract}
\maketitle
\tableofcontents
\section{Introduction and main results}
\subsubsection*{Background and setting} The concept of real valued Sobolev functions defined on a metric measure space $(\X,\sfd_\X,\mm_\X)$ is by now well understood. Given an exponent $p\in[1,\infty)$ the space of functions $f:\X\to\R$ having `distributional differential in $L^p(\X)$ in a suitable sense' is denoted by $\Sob^p(\X)$. To each $f\in\Sob^p(\X)$ one associates the function $|D f|\in L^p(\X)$, called \emph{minimal weak upper gradient}, which in the smooth setting coincides with the modulus of the distributional differential (see \cite{Cheeger00} and \cite{Shanmugalingam00}, \cite{AmbrosioGigliSavare11}).

Inspired by the work of Weaver \cite{Weaver01}, in \cite{Gigli14} the first author built the theory of $L^p$-normed modules and gave a notion of \emph{differential} $\d f$ for maps $f\in\Sob^p(\X)$ in that framework: by definition, $\d f$ is an element of the so called \emph{cotangent $L^p$-normed module} $L^p(T^*\X)$ and has the property that its pointwise norm coincides $\mm_\X$-a.e.\ with $|Df|$. We remark that the linear structure of the space $\Sob^p(\X)$, a consequence of the fact that the target space $\R$ is a vector space, plays a key role in the construction.

We now turn to the case of metric-valued Sobolev maps. Let $(\X,\sfd_\X,\mm_\X)$ be a metric measure space as before and let $(\Y,\sfd_\Y)$ be a metric space which shall be assumed to be complete and separable. We shall also fix $p=2$ for simplicity. There are various possible definitions of the concept of Sobolev maps from $\X$ to $\Y$; here we shall work with the one based on post-composition (see \cite{HKST15} for historical remarks): we say that $f\in \Sob^2(\X;\Y)$ provided there is $G\in L^2(\X)$ such that for any $\varphi:\Y\to\R$ Lipschitz we have $\varphi\circ f\in \Sob^2(\X)$ with
\[
|D(\varphi\circ f)|\leq \Lip(\varphi)\,G\quad\mm_\X-a.e..
\]
The least such $G$ is then denoted $|D f|$ and called the \emph{minimal weak upper gradient} of the map $f$. Notice that since $\Y$ has no linear structure, the set $\Sob^2(\X;\Y)$ is not a vector space in general.

\bigskip

The question we address in this paper is the following: in analogy with the fact that `behind' the minimal weak upper gradient $|Df|$ of a real-valued Sobolev map there is an abstract differential $\d f$,  \emph{does there exist a notion of {differential} for a metric-valued Sobolev map?}

\bigskip

Before turning to the (positive) answer to this question, let us motivate our interest in the problem, which goes beyond the mere desire of generalization. In the celebrated paper \cite{ES64}, Eells and Sampson  proved Lipschitz regularity for harmonic maps between Riemannian manifolds when the target $N$ has non-positive curvature and is simply connected, and the Lipschitz estimate is given in terms of a lower Ricci curvature bound and an upper dimension bound on the source manifold $M$. A key point in their proof is the establishment of the now-called Bochner-Eells-Sampson formula for maps $f:M\to N$ which we shall write as
\begin{equation}
\label{eq:BES}
\Delta\frac{|\d f|^2}{2}\geq \nabla f(\Delta f)+K|\d f|^2,
\end{equation}
where $|\d f|$ is the Hilbert-Schmidt norm of the differential of $f$ and $K\in\R$ is a lower bound for the Ricci curvature of $M$ (let us remain vague about the meaning of $\nabla f(\Delta f)$). A direct consequence of \eqref{eq:BES} is that if $f$ is harmonic, then
\begin{equation}
\label{eq:BES2}
\Delta\frac{|\d f|^2}{2}\geq K|\d f|^2.
\end{equation}
This bound and Moser's iteration technique are sufficient to show that  $|\d f|$ is locally bounded from above in the domain of definition of $f$, thus showing the local Lipschitz regularity of $f$ (the upper dimension bound for $M$ enters into play in the constants appearing in Moser's argument).

Since the Lipschitz regularity of harmonic functions does not depend on the smoothness of $M$ and $N$ but only in the stated curvature bounds, it is natural to ask whether the same results hold assuming only the appropriate curvature bounds on the source and target space, without any reference to smoothness. Efforts in this direction have been made by Gromov-Schoen in \cite{GS92}, by Korevaar-Shoen in \cite{KS93} and by Zhang-Zhu in \cite{ZZ18}. The most general result is in \cite{ZZ18}, where the authors consider the case of source spaces which are finite-dimensional Alexandrov spaces with (sectional) curvature bounded from below and targets which are ${\sf CAT}(0)$ spaces. Still, given Eells-Sampson's result the natural synthetic setting appears to be that of maps from a ${\sf RCD}(K,N)$ space to a ${\sf CAT}(0)$ space;  as of today, this appears to be out of reach. Let us remark that in none of these 3 papers has inequality \eqref{eq:BES} been written down explicitly; in \cite{KS93} and \cite{ZZ18} ``only" a form of \eqref{eq:BES2} for harmonic maps has been established (in  \cite{GS92} the argument was different and based on Almgren's frequency function).

The present manuscript aims at being a first step in the direction of obtaining \eqref{eq:BES} for maps from ${\sf RCD}$ spaces to ${\sf CAT}(0)$ ones (see also \cite{GT18}): if succesful, this research project easily implies the desired Lipschitz regularity for harmonic maps and at the same time improves the understanding of the subject even in previously studied non-smooth settings.  The overall program is definitely ambitious, but we believe that even intermediate steps like the current manuscript have an intrinsic interest: see in particular the `review' of Kirchheim's notion of metric differential in Section \ref{se:kir}.

\bigskip The very first step to tackle in order to write down \eqref{eq:BES} is to understand what ``$\d f$" is. As stated, this is our goal in this manuscript. Let us informally describe the key concept in this work (the precise definitions will be given in Sections \ref{se:pre} and \ref{s:diff_Sob_map}).

\bigskip\noindent\subsubsection*{Differential of Sobolev maps} Given a Sobolev map $u\in {\rm S^2}(\X;\Y)$ between a metric measure space $(\X,\sfd_\X,\mm)$ and a complete separable metric space $(\Y,\sfd_\Y)$, we consider the metric measure space $(\Y,\sfd_\Y,\mu)$, where $\mu=:u_\ast(|Du|^2\mm)$. Then we define the differential $\d u$ of $u$ as an operator $$\d u :\,L^0(T\X)\to (u^\ast L^0_\mu(T^\ast\Y))^\ast$$ satisfying
\begin{equation}\label{differential}
\langle u^\ast \d f,\d u(V)\rangle =V(\d (f\circ u))\quad\mm\textrm{-a.e.}
\end{equation}
for every $f\in \Sob^2(\Y,\sfd_\Y,\mu)$ and $V\in L^0(T\X)$ (Definition \ref{def:differential_Sobolev_maps}).

The particular choice of measure $\mu$ is important: it ensures that for $f\in \Sob^2(\Y,\sfd_\Y,\mu)$ the pullback function $u^*f:=f\circ u$ belongs to $\Sob^2(\X,\sfd_\X,\mm)$ with 
\begin{equation}
\label{pbdiff}
|D(f\circ u)|\leq |Df|\circ u\,|D u|,
\end{equation}
see Proposition \ref{prop:compo} for the precise formulation. Once this is established,  the differential of $u$ can be defined by taking the appropriate adjoint of the map $\d f\mapsto \d(f\circ u)$, as in \eqref{differential}. Let us emphasise that on the right hand side of the crucial bound \eqref{pbdiff} there is the product of two `weak' objects: this makes the inequality non-trivial.

\bigskip Once the definition is given we verify that it is compatible, and thus generalizes, previously existing notions of differentials in the non-smooth setting. All our discussion is made for the Sobolev exponent $p=2$, but obvious modifications generalise all the results to the case $p\in(1,\infty)$.

\section{Preliminaries}\label{se:pre}
To keep the presentation short we assume the reader is familiar with the concept of Sobolev functions on a metric measure space (\cite{Cheeger00}, \cite{Shanmugalingam00}, \cite{AmbrosioGigliSavare11}, \cite{AmbrosioGigliSavare11-3}) and with that of $L^0$-normed modules and differentials of real valued Sobolev maps (\cite{Gigli14}, \cite{Gigli17}).

Here we only recall those concepts we shall use most frequently. Let us fix a complete, separable metric space $(\X,\sfd_\X)$ and a  non-negative and non-zero Radon measure $\mm$ giving finite mass to bounded sets. We shall denote by $\Lip(f)$ the (global) Lipschitz constant of a function, by $\LIP(\X),\LIP_{bs}(\X),  \LIP_{bd}(\X)$ the space of Lipschitz functions, Lipschitz functions with bounded support, and functions which are Lipschitz on bounded sets, respectively. We also denote by $\lip_a(f):\X\to[0,\infty]$ the asymptotic Lipschitz constant, defined by
\[
\lip_a(f)(x):=\lims_{y,z\to x}\frac{|f(y)-f(z)|}{\sfd_\X(y,z)}\quad\text{ if $x$ is not isolated, $0$ otherwise.}
\] 
Then we define:
\begin{definition}[The Sobolev class $\Sob^2(\X)$]\label{def:sob}
We say that  $f\in \Sob^2(\X)$ provided there is a function $G\in L^2(\mm)$ and a sequence $(f_n)\subset \LIP_{bd}(\X)$ converging to $f$ in $L^0(\mm)$  such that $(\lip_a(f_n))$ weakly converges to $G$ in $L^2(\mm)$.
\end{definition}
With respect to the approach in \cite{AmbrosioGigliSavare11}, \cite{AmbrosioGigliSavare11-3} here the difference is in the topology used in the relaxation procedure. The fact that our approach is equivalent to the one in \cite{AmbrosioGigliSavare11}, \cite{AmbrosioGigliSavare11-3} follows from the $L^0$-stability of weak upper gradients granted by the approach via test plans in conjunction with a cut-off argument.

For $f\in \Sob^2(\X)$ we recall that there is a minimal, in the $\mm$-a.e.\ sense, non-negative function $G\in L^2(\mm)$ for which the situation in Definition \ref{def:sob} occurs. Such $G$ is denoted $|D f|$ and called minimal weak upper gradient. It is then easy to check that:
\begin{equation}
\label{eq:optlip}
\text{$\forall f\in \Sob^2(\X)$ there is $(f_n)\subset \LIP_{bd}(\X)$ $\mm$-a.e.\ converging to $f$ such that $\lip_a(f_n)\to |D f|$ in $L^2(\mm)$.}
\end{equation}
From the minimal weak upper gradients one can `extract' a notion of differential:
\begin{theorem}[Cotangent module and differential]\label{thm:defdif} With the above notation and assumptions, there is a unique (up to unique isomorphism) couple $(L^0(T^*\X),\d)$ with $L^0(T^*\X)$ being a $L^0(\mm)$ normed module, $\d:\Sob^2(\X)\to L^0(T^*\X)$ linear and such that: $|\d f|=|Df|$ $\mm$-a.e.\ for every $f\in \Sob^2(\X)$ and $\{\d f:f\in\Sob^2(\X)\}$ generates $L^0(T^*\X)$.
\end{theorem}
When we want to emphasise the role of the chosen measure, we shall write $(L^0_\mm(T^*\X),\d_\mm)$ in place of $(L^0(T^*\X),\d)$. Among the various properties of the differential, we shall frequently use its \emph{locality}:
\[
\d f=\d g\quad\mm-a.e.\ on\ \{f=g\},\qquad\forall f,g\in \Sob^2(\X).
\]

\bigskip
Let us now recall few facts about \emph{pullback} of modules:
\begin{theorem}[Pullback]\label{thm:pb}
Let $(\X,\sfd_\X,\mm_\X),(\Y,\sfd_\Y,\mm_\Y)$ be metric measure spaces as above, $u:\X\to\Y$ such that $u_*\mm_\X\ll\mm_\Y$ and $\mathscr M$ an $L^0(\mm_\Y)$-normed module. Then there is a unique (up to unique isomorphism) couple $(u^*\mathscr M,[u^*])$ such that $u^*\mathscr M$ is a $L^0(\mm_\X)$-normed module and $[u^*]:\mathscr M\to u^*\mathscr M$ is linear, continuous and such that $|[u^*v]|=|v|\circ u$ $\mm_\X$-a.e.\ for every $v\in\mathscr M$ and $\{[u^*v]:v\in\mathscr M\}$ generates $u^*\mathscr M$.
\end{theorem}
The module $u^*\mathscr M$ is called the pullback module and $[u^*]$ the pullback map. It can be directly checked by the uniqueness part of Theorem \ref{thm:pb} that 
\begin{equation}
\label{eq:pbfct}
\text{if }\quad \text{$u_*\mm_\X\ll\mm_\Y$}\quad\text{then}\quad  u^*L^0(\mm_\Y)\sim L^0(\mm_\X) \quad\text{ via the map }\quad [u^*f]\mapsto f\circ u.
\end{equation}

The pullback has the following universal property, which we shall frequently use: 
\begin{proposition}[Universal property of the pullback]\label{prop:univprop}
With the same notation and assumptions as in Theorem \ref{thm:pb} above, let $V\subset \mathscr M$ a generating subspace, $\mathscr N$ a $L^0(\mm_\X)$-normed module and $T:V\to \mathscr N$ a linear map such that $ |T(v)|\leq f|v|\circ u$ $\mm_\X$-a.e. $\forall v\in V$ for some $f \in L^0(\mm_\X)$. Then there exists a unique $L^0(\mm_\X)$-linear and continuous map $\tilde T:u^*\mathscr M\to\mathscr N$ such that $\tilde T([u^*v])=T(v)$ for every $v\in V$ and this map satisfies 
\begin{equation}
\label{eq:pb1}
 |\tilde T(w)|\leq f | w|\quad\mm_\X-a.e.\qquad\forall w\in u^*\mathscr M.
\end{equation}
In particular, if $T:\mathscr M_1\to\mathscr M_2$ is a $L^0(\mm_\Y)$-linear and continuous map satisfying $ |T(v)|\leq g|v|$ $\mm_\Y$-a.e. $\forall v\in\mathscr M_1$, for some $g \in L^0(\mm_\Y)$, applying the above to the map $\mathscr M_1\ni v\mapsto [u^*T(v)]\in u^*\mathscr M_2$ we deduce that there exists a unique $L^0(\mm_\X)$-linear and continuous map $u^*T:u^*\mathscr M_1\to u^*\mathscr M_2$ making the diagram
\[
\begin{tikzcd}
\mathscr M_1  \arrow[r,"T"] \arrow[d,"{[u^*]}"'] & \mathscr M_2  \arrow[d,"{[u^*]}"]\\
u^*\mathscr M_1  \arrow[r,"u^*T "'] & u^*\mathscr M_2 
\end{tikzcd}
\]
commute and such map satisfies
\begin{equation}
\label{eq:pb2}
|u^*T(w)|\leq g\circ u|w|\quad\mm_\X-a.e.\qquad\forall w\in u^*\mathscr M_1.
\end{equation}
\end{proposition}
These properties of pullbacks have been studied in \cite{Gigli14}, \cite{Gigli17} for maps satisfying $u_*\mm_\X\leq C\mm_\Y$, but if one is only interested in $L^0$-modules the theorems above are easily seen to hold with small modifications.

\bigskip

Finally, let us present a simple construction that we shall frequently use. Let $E\subset\X$ be Borel, put $\nu:=\mm\restr E$ and let $\mathscr M$ be a $L^0(\nu)$-normed module. To such a module we can canonically associate a $L^0(\mm)$-normed module, called \emph{extension} of $\mathscr M$ and denoted by ${\rm Ext}(\mathscr M)$, in the following way. First of all we notice that we have a natural projection/restriction operator ${\rm proj}:L^0(\mm)\to L^0(\nu)$ given by passage to the quotient up to equality $\nu$-a.e.\ and a natural `extension' operator ${\rm ext}:L^0(\nu)\to L^0(\mm)$ which sends $f\in L^0(\nu)$ to the function equal to $f$ $\mm$-a.e.\ on $E$ and to $0$ on $\X\setminus E$. Then for a generic $L^0(\nu)$-normed  module $\mathscr M$ we put ${\rm Ext}(\mathscr M):=\mathscr M$ as set,  multiplication of $v\in {\rm Ext}(\mathscr M) $ by $f\in L^0(\mm)$ is defined as ${\rm proj}(f)v\in \mathscr M={\rm Ext}(\mathscr M)$ and the pointwise norm as ${\rm ext}(|v|)\in L^0(\mm)$. We shall denote by ${\rm ext}:\mathscr M\to {\rm Ext}(\mathscr M)$ the identity map and notice that in a rather trivial way we have
\begin{equation}
\label{eq:extdual}
{\rm Ext}(\mathscr M^*)\sim {\rm Ext}(\mathscr M)^*\qquad\text{ via the coupling }\qquad{\rm ext}(L)\big({\rm ext}(v)\big):={\rm ext}(L(v)).
\end{equation}
In what follows we shall  always implicitly make this identification.

\section{Differential of metric-valued Sobolev maps}\label{s:diff_Sob_map}
Throughout this manuscript $(\X,\sfd_\X,\mm)$ will always denote a complete separable metric space endowed with a non-negative and non-zero Radon measure which is finite on bounded sets; $(\Y,\sfd_\Y)$ denotes a complete (not necessarily separable) metric space.

\begin{definition}[Metric valued Sobolev map]\label{def:sobmap}
The set $\Sob^2(\X,\Y)$ is the collection of all Borel maps $u:\X\to\Y$ which are essentially separably valued (i.e.\ there is a null set $N\subset\X$ so that $u(\X\setminus N)\subset\Y$ is separable)
for which there is $G\in L^2(\X,\mm)$, $G\geq 0$ such that for any $f\in\LIP(\Y)$ it holds $f\circ u\in \Sob^2(\X)$ and
\begin{equation}
\label{eq:defug}
|\d (f\circ u)|\leq \Lip(f)G\qquad\mm-a.e..
\end{equation}
The least, in the $\mm$-a.e.\ sense, function $G$ for which the above holds will be denoted $|D u|$.
\end{definition}
Notice that for $u\in\Sob^2(\X,\Y)$ the class of $G\in L^2(\X)$ for which \eqref{eq:defug} holds is a closed lattice, hence a $\mm$-a.e.\ minimal one exists and the definition of $|D u|$ is well posed.

Our study of functions in $\Sob^2(\X,\Y)$ begins  with the following basic lemma:
\begin{lemma}\label{le:first}
Let $u\in \Sob^2(\X,\Y)$ and $f\in\LIP(\Y)$. Then $f\circ u\in\Sob^2(\X)$ with
\begin{equation}
\label{eq:lipa}
|\d(f\circ u)|\leq \lip_a(f)\circ u\,|D u|\qquad\mm-a.e..
\end{equation}
\end{lemma}
\begin{proof}
Replacing if necessary $\Y$ with a closed separable subset containing almost all the image of $u$ we can assume that $\Y$ is separable. Then let $(y_n)\subset \Y$ be countable and dense and for $r\in \Q$, $r>0$,  let $f_{r,n}\in\LIP(\Y)$ be a McShane extension of $f\restr{B_r(y_n)}$, i.e.\ a Lipschitz map defined on the whole $\Y$ which coincides with $f$ on $B_r(y_n)$ and such that $\Lip(f_{r,n})=\Lip(f\restr{B_r(y_n)})$. Then from \eqref{eq:defug} and the locality of the differential we see that
\[
|\d (f\circ u)|\leq \Lip(f\restr{B_r(y_n)})|D u|\qquad \mm-a.e.\ on\ u^{-1}(B_r(y_n)).
\]
Since for every $y\in\Y$ we have $\lip_a(f)(y)=\inf\Lip(f\restr{B_r(y_n)})$, where the $\inf$ is taken among all $n,r$ such that $y\in B_{r}(y_n)$, the conclusion follows.
\end{proof}
Let us fix  $u\in {\rm S}^2(\X,\Y)$ and equip the target space $\Y$ with the finite Radon measure
\[
\mu:=u_\ast(|D u|^2\mm).
\]
Notice that for $f\in L^0(\Y,\mu)$ the function $f\circ u$ is not well-defined up to equality $\mm$-a.e.\ in the sense that if $f=\tilde f$ $\mu$-a.e., then not necessarily $f\circ u=\tilde f\circ u$ $\mm$-a.e.. Still, we certainly have $f\circ u=\tilde f\circ u$ $\mm$-a.e.\ on $\{|D u|>0\}$ and for this reason we have  $f\circ u\,|D u|=\tilde f\circ u\,|D u|$ $\mm$-a.e., i.e.\ the map $f\mapsto f\circ u\,|D u|$ is well defined from $L^0(\Y,\mu)$ to $L^0(\X,\mm)$. Then the identity $\int  \big|f\circ u\,|D u|\big|^2\,\d\mm=\int |f|^2\,\d\mu$ shows that
\begin{equation}
\label{eq:cont}
L^2(\Y,\mu)\ni f\quad\mapsto\quad f\circ u\,|D u| \in L^2(\X,\mm)\qquad\text{ is linear and continuous}.
\end{equation}
We now turn to our key basic result about pullback of Sobolev functions:
\begin{proposition}\label{prop:compo}
Let $u\in \Sob^2(\X,\Y)$, put $\mu:=u_\ast(|Du|^2\mm)$ and let $f\in {\rm S}^2(\Y,\sfd_\Y,\mu)$. Then there is $g\in\Sob^2(\X)$ such that $g=f\circ u$ $\mm$-a.e.\ on $\{|Du|>0\}$ and 
\begin{equation}
\label{eq:chain}
|\d g|\leq |\d_\mu f|\circ u|D u|\qquad\mm-a.e..
\end{equation}
More precisely, there is  $g\in\Sob^2(\X)$ and a sequence $(f_n)\subset \LIP_{bd}(\Y)$ such that
\begin{equation}
\label{eq:conv}
\begin{array}{rllrll}
f_n& \to\ f\qquad& \mu-a.e.&\qquad\qquad\lip_a(f_n)& \to\  |\d_\mu f|\qquad &\text{\rm in }L^2(\mu),\\
f_n\circ u& \to\ g\qquad& \mm-a.e.&\qquad\qquad\lip_a(f_n)\circ u|Du|& \to\  |\d_\mu f|\circ u|D u|&\text{\rm in }L^2(\mm).
\end{array}
\end{equation}
\end{proposition}
\begin{proof} Up to a truncation and diagonalization argument we can assume that $f\in L^\infty(\Y,\mu)$. Then let $(f_n)\subset \LIP_{bd}(\Y)$ be as in \eqref{eq:optlip} for $f$ and observe that since $f$ is bounded, by truncation we can assume the $f_n$'s to be uniformly bounded. Thus the first two convergences in \eqref{eq:conv} hold and, taking \eqref{eq:cont} into account we see that also the last in \eqref{eq:conv} holds. Now observe that if we can prove that $(f_n\circ u)$ has a limit $\mm$-a.e., call it $g$, then \eqref{eq:chain} would follow from Lemma \ref{le:first} above, \eqref{eq:cont} and the closure of the differential.

Let $B\subset\X$ be bounded and Borel. The functions $f_n\circ u$ are equibounded and $\mm(B)<\infty$, hence $(f_n\circ u)$ is bounded in $L^2(B,\mm\restr B)$. Thus by passing to an appropriate - not relabeled - sequence of convex combinations (which do not affect the already proven convergences in \eqref{eq:conv}) we obtain that $(f_n\circ u)$ has a strong limit in $L^2(B,\mm\restr B)$. Thus a subsequence converges $\mm$-a.e.\ on $B$ and by considering a sequence $(B_k)$ of bounded sets such that $\X=\cup_kB_k$, by a diagonalization argument we conclude the proof.
\end{proof}
Let us notice that since $\mu$ is a finite measure on $\Y$ we have  $\LIP(\Y)\subset\Sob^2(\Y,\sfd_\Y,\mu)$. Also,
\begin{equation}
\label{eq:glip}
\text{for $f\in\LIP(\Y)$ and $g\in \Sob^2(\X)$ as in Proposition \ref{prop:compo} we have }\d (f\circ u)=\d g.
\end{equation}
Indeed,  the locality of the differential gives $\d(f\circ u)=\d g$ on $\{|D u|>0\}$ and the bounds \eqref{eq:lipa} and \eqref{eq:chain} give $|\d (f\circ u)|=|\d g|=0$ $\mm$-a.e.\ on $\{|D u|=0\}$.

\bigskip

Observe that for $\nu:=\mm\restr{\{|Du|>0\}}$ we have $u_*\nu\ll\mu$, thus $u^\ast L^0_\mu(T^\ast\Y)$ is a well defined $L^0(\nu)$-normed module. Recalling the `extension' functor introduced at the end of Section \ref{se:pre},  our definition of $\d u$ is:
\begin{definition}\label{def:differential_Sobolev_maps}
	The \emph{differential} $\d u$ of $u\in {\rm S^2}(\X,\Y)$ is the operator \[ \d u:\, L^0(T\X)\to {\rm Ext}\big((u^\ast L^0_\mu(T^\ast\Y))^\ast \big)\] given as follows. 
	For $v\in L^0(T\X)$, the object $\d u(v)\in {\rm Ext}\big((u^\ast L^0_\mu(T^\ast\Y))^\ast\big)$ is characterized by the property: for every $f\in{\rm S}^2(\Y,\sfd_\Y,\mu)$ 
	and every $g\in \Sob^2 (\X,\sfd_\X,\mm)$ as in Proposition \ref{prop:compo} we have 
\begin{equation}\label{eq:def_du}
{\rm ext}\big([u^\ast \d_\mu f]\big)\big(\d u(v)\big)=\d g(v)\quad \mm-a.e..
\end{equation}
\end{definition}
We now verify that this is a good definition and check the very basic properties:
\begin{proposition}[Well posedness of the definition]\label{basic}
The differential $\d u(v)$ of $u$ in Definition \ref{def:differential_Sobolev_maps} is well-defined and the map  $\displaystyle \d u:\,L^0(T\X)\to {\rm Ext}\big((u^\ast L^0_\mu(T^\ast\Y))^\ast \big)$ is $L^0(\mm)$-linear and continuous. Moreover, it holds that
\begin{equation}
\label{eq:samenorm}
|\d u|=|Du|\qquad\mm-a.e..
\end{equation}
\begin{proof} Let $f\in \Sob^2(\Y,\sfd_\Y,\mu)$ and observe that if $g,g'\in\Sob^2(\X,\sfd_\X,\mm)$ both satisfy the properties listed in Proposition \ref{prop:compo} then the locality of the differential and the bound \eqref{eq:chain} show that $\d g=\d g'$. Hence the right hand side of \eqref{eq:def_du} depends only on $f,u,v$. Then notice that again the bound \eqref{eq:chain} gives
\[
\big|{\rm ext}\big([u^\ast \d_\mu f]\big)\big(\d u(v)\big)\big|\stackrel{\eqref{eq:def_du}}=|\d g(v)|\leq |\d g|\,|v|\stackrel{\eqref{eq:chain}}\leq |\d_\mu f|\circ u|D u|\,|v|=\big|{\rm ext}\big([u^\ast \d_\mu f]\big)\big||D u|\,|v|
\]
and thus the arbitrariness of $f\in \Sob^2(\Y,\sfd_\Y,\mu)$, Proposition \ref{prop:univprop} and property \eqref{eq:extdual} ensure that $\d u(v)$ is a well defined element of $\big({\rm Ext}(u^\ast L^0_\mu(T^\ast\Y)) \big)^\ast\sim {\rm Ext}\big((u^\ast L^0_\mu(T^\ast\Y))^\ast \big)$, as desired, with 
\begin{equation}
\label{eq:contv}
|\d u(v)|\leq |Du|\,|v|.
\end{equation}
The fact that $\d u(v)$ is $L^0(\mm)$-linear in $v$ is trivial and the bound \eqref{eq:contv} gives both continuity and the inequality $\leq$ in \eqref{eq:samenorm}. To get the other inequality let $f:\Y\to\R$ be 1-Lipschitz and  notice that since $\mu(\Y)<\infty$ we also have $f\in \Sob^2(\Y,\sfd_\Y,\mu)$. Since $u\in\Sob^2(\X,\Y)$ we have $f\circ u\in \Sob^2(\X)$ and can find  $v\in L^0(T\X)$  such that
\begin{equation}
\label{eq:gradf}
|v|=1\qquad\text{ and }\qquad \d (f\circ u)(v)=|\d (f\circ u)|\qquad \text{$\mm$-a.e.}
\end{equation}
(the existence of such $v$ follows by Banach-Alaoglu's theorem, see \cite[Corollary 1.2.16]{Gigli14}).  Moreover, let $g\in \Sob^2(\X)$ be as in Proposition \ref{prop:compo} and notice that
\[
\begin{split}
|\d(f\circ u)|&\stackrel{\eqref{eq:gradf}}=|\d (f\circ u)(v)|\stackrel{\eqref{eq:glip}}=|\d g(v)|\stackrel{\eqref{eq:def_du}}=\big|{\rm ext}\big([u^\ast \d_\mu f]\big)\big(\d u(v)\big)\big|\leq\big|{\rm ext}\big([u^\ast \d_\mu f]\big)\big|\,|\d u|\,|v|\\
&\stackrel{\eqref{eq:gradf}}=|\d_\mu f|\circ u\,|\d u|\leq |\d u|,
\end{split}
\]
having used the fact that $f$ is 1-Lipschitz in the last step. By the arbitrariness of $f$ and the very definition of $|Du|$ given in Definition \ref{def:sobmap}, this establishes $\geq$ in \eqref{eq:samenorm}. 
	\end{proof}
\end{proposition}

\section{Consistency with previously known notions}
\subsection{The case $\Y=\R$}\label{se:r}
In this section we assume $\Y=\R$ and prove that once a few natural identifications are taken into account, the  newly defined differential $\d u:\,L^0(T\X)\to{\rm Ext}\big(u^*L^0_\mu(T^*\R)\big)^*$ is `the same' as the one as defined by Theorem \ref{thm:defdif}, which for the moment we shall denote as $\ud u\in L^0(T^*\X)$.

To start with, let us observe that directly from the definitions and the chain rule 
\begin{equation}
\label{eq:chainrule}
\d(f\circ u)=f'\circ u\,\ud u\quad\mm-a.e.\qquad\forall u\in\Sob^2(\X),\ f\in C^1\cap \LIP(\R) 
\end{equation}
(see \cite[Corollary 2.2.8]{Gigli17}), we have that the class $\Sob^2(\X,\Y)$ coincides with $\Sob^2(\X)$ when $\Y=\R$ and that the two notions of minimal weak upper gradients coincide.

To continue we recall a result, obtained in \cite{GP16-2}, about the structure of Sobolev functions on weighted $\R$. For a given Radon measure $\mu$ on $\R$ we shall denote by $L^0(\R,\R^*;\mu)$ (resp.\  $L^0(\R,\R;\mu)$) the $L^0(\mu)$-normed module of maps on $\R$ with values in $\R^*$ (resp.\ $\R$)  up to equality $\mu$-a.e.. We shall instead denote by $L^0_\mu(T^*\R)$ (resp.\ $L^0_\mu(T\R)$) the cotangent (resp.\ tangent) module associated to the space $(\R,\sfd_{\rm Eucl},\mu)$. Then we have:
\begin{theorem}\label{thm:GP}
Let $\mu$ be a Radon measure on $\R$. Then there is a unique $L^0(\mu)$-linear and continuous map $P:L^0(\R,\R^*;\mu)\to L^0_\mu(T^*\R)$ such that 
\begin{equation}
\label{eq:defP}
P(Df)=\d_\mu f\qquad\forall f\in  C^1\cap \LIP(\R),
\end{equation}
where $Df:\R\to\R^*$ is the differential of $f$. Its adjoint map $\iota:L^0_\mu(T\R)\to L^0(\R,\R;\mu)$ is an isometry. In particular, $L^0_\mu(T\R)$ is separable.
\end{theorem}
Let now $u\in \Sob^2(\X)$, put $\mu:=u_*(|\d u|^2\mm)$ and consider the $L^0(\mm)$-normed module ${\rm Ext}(u^*L^0_\mu(T^*\R))$. From the separability of $L^0_\mu(T\R)$ provided by Theorem \ref{thm:GP} above, the characterisation of the dual of the pullback obtained in  \cite[Theorem 1.6.7]{Gigli17} and \eqref{eq:extdual} we see that
\[
{\rm Ext}(u^*L^0_\mu(T\R))\sim {\rm Ext}(u^*L^0_\mu(T^*\R))^* \quad\text{via the coupling}\quad {\rm ext}([u^*L])\big({\rm ext}([u^*v])\big):={\rm ext}(L(v)\circ u).
\]
Hence in the present situation we shall think of $\d u$ as a map from $L^0(T\X)$ to ${\rm Ext}(u^*L^0_\mu(T\R))$.

Now put $\nu:=\nchi_{\{|Du|>0\}}\mm$ as before and consider the $L^0(\nu)$-linear and continuous operators
\[\begin{split}
u^* P:\,u^*L^0(\R,\R^*;\mu)\longrightarrow u^*L^0_\mu(T^*\R),\qquad\qquad u^*\iota:\,u^*L^0_\mu(T\R)\longrightarrow u^*L^0(\mu)\stackrel{\eqref{eq:pbfct}}\sim L^0(\nu)
\end{split}\]
defined via  the universal property of the pullback module given in Proposition \ref{prop:univprop}. It is then clear that $u^*\iota$ is the adjoint of $u^*P$, thus  from \eqref{eq:defP} we see that 
\begin{equation}\label{eq:charact_tilde_iota}
[u^*Df]\big(u^*\iota(V)\big)=[u^*\d_\mu f](V)\quad\nu-a.e.\qquad\text{ for every }V\in u^*L^0_\mu(T\R),\ f\in C^1_c(\R).
\end{equation}
Finally, noticing that ${\rm ext}:u^*L^0_\mu(T\R)\to {\rm Ext}(u^*L^0_\mu(T\R))$ is invertible,  we define $\mathcal I:{\rm Ext}(u^*L^0_\mu(T\R))\to L^0(\mm)$ as
\begin{equation}
\label{eq:defI}
\mathcal I:={\rm ext}\circ u^*\iota\circ{\rm ext}^{-1}.
\end{equation}
Then we have:
\begin{theorem} With the above notation and assumptions we have $|\d u|=|\ud u|$ $\mm$-a.e.\ and 
\begin{equation}
\label{eq:samediff}
\mathcal I(\d u(v))=\ud u(v)\quad\mm-a.e.\qquad\forall v\in L^0(T\X).
\end{equation}
\end{theorem}
\begin{proof} The identity $|\d u|=|\ud u|$ follows from \eqref{eq:samenorm} and the already noticed fact that for $u\in\Sob^2(\X)=\Sob^2(\X,\R)$ the two notions of minimal weak upper gradients underlying the two spaces coincide. 

We turn to \eqref{eq:samediff}. For $f\in C^1_c(\R)$ let us denote by $D f:\R\to\R^*$ its differential and by $f':\R\to\R$ its derivative. Clearly, up to identifying $\R$ and $\R^*$ via the Riesz isomorphism these two objects coincide and thus checking first the case $h=[u^*g]$ we easily get that
\begin{equation}
\label{eq:riesz}
f'\circ u\,h={\rm ext}[u^*Df](h)\quad\mm-a.e.\qquad\forall h\in {\rm Ext}\big(u^\ast L^0(\mu)  \big)\stackrel{\eqref{eq:pbfct}}\sim {\rm Ext}\big( L^0(\nu)  \big)\subset L^0(\mm).
\end{equation}
Then for $g$ as in Proposition \ref{prop:compo} we have
\[
\begin{split}
f'\circ u\,\mathcal I(\d u(v))&\stackrel{\eqref{eq:riesz}}={\rm ext}[u^*Df]\,\mathcal I(\d u(v))\stackrel{\eqref{eq:defI},\eqref{eq:extdual}}={\rm ext}\big([u^*Df]\big(u^*\iota\big({\rm ext}^{-1}(\d u(v))\big)\big)\big)\\
&\stackrel{\eqref{eq:charact_tilde_iota}}={\rm ext}\big([u^*\d_\mu f]\big({\rm ext}^{-1}(\d u(v))\big)\big)\stackrel{\eqref{eq:extdual}}={\rm ext}([u^*\d_\mu f])(\d u(v))\stackrel{\eqref{eq:def_du}}=\d g(v)\\
&\stackrel{\eqref{eq:glip}}=\d (f\circ u)(v)\stackrel{\eqref{eq:chainrule}}=f'\circ u\,\ud u(v).
\end{split}
\]
Since the space $\{f'\circ u:f\in C^1_c(\R)\}$ generates $L^0(\mm)$, this is sufficient to establish \eqref{eq:samediff}.
\end{proof}
\subsection{The case $u$ of bounded deformation}\label{se:bd}
In this section we shall assume that also $(\Y,\sfd_\Y)$ carries a non-negative Radon measure $\mm_\Y$ which gives finite mass to bounded sets and study the differential of a map $u\in \Sob^2(\X,\Y)$ which is also of \emph{bounded deformation}. Recall that the latter means that $u$ is Lipschitz and for some $C>0$ it holds $u_*\mm_\X\leq C\mm_\Y$,
where we denote $\mm_\X:=\mm$ for the sake of clarity. For such $u$ it is easy to prove that 
\[
f\in\Sob^2(\Y)\quad\Rightarrow\quad f\circ u\in \Sob^2(\X)\qquad\text{with}\qquad |\d (f\circ u)|\leq \Lip(u)|\d f|\circ u\quad\mm_\X-a.e..
\] 
Then a notion of differential $\hat\d u:\,L^2(T\X)\to\big(u^*L^2_{\mm_\Y}(T^*\Y)\big)^*$ can be defined by the formula
\begin{equation}
\label{eq:dhat}
[u^*\d_{\mm_\Y}f](\hat\d u(v)):=\d(f\circ u)(v)\quad\mm_\X-a.e.\qquad\forall f\in\Sob^2(\Y,\sfd_\Y,\mm_\Y),\ v\in L^2(T\X),
\end{equation}
see \cite[Proposition 2.4.6]{Gigli17}. In this section we study the relation between $\hat \d u$ and $\d u$. We start noticing that  the definition of $|D u|$ trivially gives $|D u|\leq \Lip(u)$ $\mm_\X$-a.e., so we have 
\begin{equation}
\label{eq:mumm}
\mu=u_*(|Du|^2\mm_\X)\leq \Lip^2(u)u_*\mm_\X\leq C\Lip^2(u)\mm_\Y.
\end{equation}
Also, let us prove the following general statement:
\begin{lemma}\label{le:changemeas}
Let $\mu_1,\mu_2$ be two non-negative and non-zero Radon measures on the complete space $(\Y,\sfd_\Y)$ with $\mu_1\leq\mu_2$. Then  $\Sob^2(\Y,\sfd_\Y,\mu_2)\subset\Sob^2(\Y,\sfd_\Y,\mu_1)$ and there is a unique $L^0(\mu_2)$-linear and continuous map $P:L^0_{\mu_2}(T^*\Y)\to {\rm Ext}(L^0_{\mu_1}(T^*\Y))$ such that
\[
P(\d_{\mu_2}f)={\rm ext}(\d_{\mu_1}f)\qquad\forall f\in \Sob^2(\Y,\sfd_\Y,\mu_2),
\]
and it satisfies  $|P(\omega)|\leq|\omega|$ $\mu_2$-a.e.\ for every $\omega\in L^0_{\mu_2}(T^*\Y)$, where here the `extension' operator acts from $L^0(\mu_1)$- to $L^0(\mu_2)$- normed modules.
\end{lemma}
\begin{proof}
The assumption $\mu_1\leq\mu_2$ ensures that the topologies of $L^2(\mu_2),L^0(\mu_2)$ are stronger than those of $L^2(\mu_1),L^0(\mu_1)$ respectively. Thus both the  inclusion $\Sob^2(\Y,\sfd_\Y,\mu_2)\subset\Sob^2(\Y,\sfd_\Y,\mu_1)$ and the bound ${\rm ext}(|\d_{\mu_1}f|)\leq  |\d_{\mu_2}f|$ $\mu_2$-a.e.\ for every $f\in \Sob^2(\Y,\sfd_\Y,\mu_2)$ follow from  Definition \ref{def:sob}. To conclude apply, e.g., Proposition \ref{prop:univprop} with $u:={\rm Identity}$ and $T(\d_{\mu_2}f):={\rm ext}(\d_{\mu_1}f)\in {\rm Ext}(L^0_{\mu_1}(T^*\Y))$. 
\end{proof}
Applying this lemma to the case under consideration we get:
\begin{proposition}\label{prop:pi} Assume that $u:\X\to\Y$ is of bounded compression. Then with the above notation there is a unique $L^0(\mm_\Y)$-linear and continuous map $\pi:L^0_{\mm_\Y}(T^*\Y)\to {\rm Ext}(L^0_{\mu}(T^*\Y))$ such that $\pi(\d_{\mm_\Y}f)={\rm ext}(\d_\mu f)$ for every $f\in\Sob^2(\Y,\sfd_\Y,\mm_\Y)$ (the extension operator being intended from $L^0(\mu)$- to $L^0(\mm_\Y)$- normed modules) and it satisfies $|\pi(\omega)|\leq|\omega|$ $\mm_\Y$-a.e.\ for every $\omega\in L^0_{\mm_\Y}(T^*\Y)$.

Moreover, for any $ f\in \Sob^2(\Y,\sfd_\Y,\mm_\Y)$ and $g\in\Sob^2(\X)$ as in Proposition \ref{prop:compo} we have
\begin{equation}
\label{eq:diffuguali}
\d g=\d(f\circ u).
\end{equation}
\end{proposition}
\begin{proof} The first part of the statement follows from  Lemma \ref{le:changemeas} and \eqref{eq:mumm}.  To prove \eqref{eq:diffuguali}  notice that thanks to the locality of the differential we know that \eqref{eq:diffuguali} holds $\mm_\X$-a.e.\ on $\{|D u|>0\}$, while \eqref{eq:chain} shows that $\d g=0$ $\mm_\X$-a.e.\ on $\{|D u|=0\}$, hence to conclude it is sufficient to prove that $|\d (f\circ u)|=0$ $\mm_\X$-a.e.\ on $\{|D u|=0\}$. To see this, let $(f_n)\subset \LIP_{bd}(\Y)$ be such that $(f_n),(\lip_a(f_n))$ converge to $f,|\d_{\mm_\Y}f|$  $\mm_\Y$-a.e.\ and in $L^2(\mm_\Y)$ respectively. Then the assumption $u_*\mm_\X\leq C\mm_\Y$ grants that $(f_n\circ u),\big(\lip_a(f_n)\circ u\big)$ converge to $f\circ u, |\d_{\mm_\Y} f|\circ u$  $\mm_\X$-a.e.\ and in $L^2(\mm_\X)$ respectively. Hence passing to the limit in \eqref{eq:lipa} we conclude that $|\d (f\circ u)|=0$ $\mm_\X$-a.e.\ on $\{|D u|=0\}$, as desired.
\end{proof}
It is readily verified that the map sending $[u^*{\rm ext}(\omega)]$ to ${\rm ext}([u^*\omega])$ is an isomorphism from $u^*{\rm Ext}(L^0_\mu(T^*\Y))$ to ${\rm Ext}(u^*L^0_\mu(T^*\Y))$, hence from Proposition \ref{prop:pi} above and the universal property of the pullback stated in Proposition \ref{prop:univprop} we see that there is a unique $L^0(\mm_\X)$-linear and continuous map $u^*\pi:u^*L^0_{\mm_\Y}(T^*\Y)\to {\rm Ext}(u^*L^0_\mu(T^*\Y))$ such that
\begin{equation}
\label{eq:defpi}
u^*\pi([u^*\d_{\mm_\Y} f])={\rm ext}[u^*\d_\mu f]\qquad\forall f\in\Sob^2(\Y,\sfd_\Y,\mm_\Y)
\end{equation}
and such map satisfies
\begin{equation}
\label{eq:normpi}
|u^*\pi(\omega)|\leq |\omega|\quad\mm_\X-a.e.\qquad\forall \omega\in u^*L^0_{\mm_\Y}(T^*\Y).
\end{equation}
Then denoting by  $(u^*\pi)^*:\big({\rm Ext}(u^*L^0_\mu(T^*\Y))\big)^*\to \big(u^*L^0_{\mm_\Y}(T^*\Y)\big)^*$ the adjoint of $u^*\pi$ we have:

\begin{theorem}\label{thm:ineq_diff} With the above notation and assumptions we have
\begin{equation}\label{eq:ineq_diff_claim}
\hat\d u(v)=(u^*\pi)^*\big(\d u(v)\big)\qquad\forall v\in L^0(T\X)
\end{equation}
and 
\begin{equation}\label{eq:ineq_diff_conseq}
\big|\hat\d u(v)\big|\leq\big|\d u(v)\big|
\quad\mm_\X\text{-a.e.\ on }\X\qquad \forall v\in L^0(T\X).
\end{equation}
\end{theorem}
\begin{proof} Let $f\in{\rm S}^{2}(\Y,\sfd_\Y,\mm_\Y)$ and notice that
\[
[u^*\d_{\mm_\Y} f](\hat \d u(v))\stackrel{\eqref{eq:dhat}}=\d(f\circ u)(v)\stackrel{\eqref{eq:diffuguali}}=\d g(v)\stackrel{\eqref{eq:def_du}}={\rm ext}[u^*\d_\mu f](\d u(v))\stackrel{\eqref{eq:defpi}}=(u^*\pi)([u^*\d_{\mm_\Y} f])(\d u(v)).
\]
Since elements of the form $[u^*\d_{\mm_\Y} f]$ generate $u^*L^0_{\mm_\Y}(T^*\Y)$, this is sufficient to prove \eqref{eq:ineq_diff_claim}. Now observe that by duality \eqref{eq:normpi} yields $|(u^*\pi)^*(V)|\leq |V|$ $\mm_\X$-a.e.\ for every $V\in \big({\rm Ext}(u^*L^2_\mu(T^*\Y))\big)^*$, hence \eqref{eq:ineq_diff_conseq} follows from \eqref{eq:ineq_diff_claim}.
\end{proof}
Equality in \eqref{eq:ineq_diff_conseq} can be obtained under appropriate assumptions on either $\X$ or $\Y$:
\begin{proposition}\label{thm:equal_diff}
Suppose that either $W^{1,2}(\X,\sfd_\X,\mm_\X)$ or $W^{1,2}(\Y,\sfd_\Y,\mu)$ is reflexive. Then
\[
\big|\hat\d u(v)\big|=\big|\d u(v)\big|\;\;\;\text{holds }\mm_\X\text{-a.e.\ }
\quad\text{ for every }v\in L^0(T\X).
\]
\end{proposition}
\begin{proof}\ \\
\noindent{\bf  $W^{1,2}(\X,\sfd_\X,\mm)$ is reflexive.} By inequality \eqref{eq:ineq_diff_conseq} and a density argument to conclude it is sufficient to show that for any $f\in L^\infty\cap \Sob^{2}(\Y,\sfd_\Y,\mu)$, $g\in \Sob^2(\X)$ as in Proposition \ref{prop:compo} and $v\in L^\infty(T\X)$ with bounded support it holds
\begin{equation}
\label{eq:toprove}
\d g(v)\leq |\d_\mu f|\circ u|\hat\d u(v)| \qquad\mm-a.e..
\end{equation}
Let us observe that \eqref{eq:ineq_diff_conseq} and the very definition of $|\d u|$ give $|\hat\d u(v)|\leq |\d u(v)|\leq |\d u||v|$ $\mm$-a.e., hence  the $\mm$-a.e.\ value of $G\circ u|\hat\d u(v)|$ is independent on the $\mu$-a.e.\ representative  of $G$, and the right hand side of \eqref{eq:toprove} is well defined $\mm$-a.e.\ (and equal to 0 $\mm$-a.e.\ on $\{|\d u|=0\}$). The trivial bound 
\[
\int |G|^2\circ u|\hat\d u(v)|^2\,\d\mm\leq \int|G|^2\circ u|\d u|^2|v|^2\,\d\mm\leq\||v|\|^2_\infty\int |G|^2\,\d u_*(|\d u|^2\mm) 
\]
shows that 
\begin{equation}
\label{eq:perdopo}
L^2(\Y,\mu)\ni G\quad\mapsto\quad G\circ u|\hat \d u(v)|\in L^2(\X,\mm)\qquad\text{ is linear and continuous}.
\end{equation}
Now fix $f,v$ as in \eqref{eq:toprove}, let $\eta\in\LIP(\X)$ be identically 1 on the support of $v$ and with bounded support, $(f_n)\subset \LIP_{bd}(\Y)$ be as in \eqref{eq:optlip} for the space $(\Y,\sfd_\Y,\mu)$ and notice that since we assumed $f$ to be bounded, up to a truncation argument we can assume the $f_n$'s to be equibounded. Thus the functions $f_n\circ u$ are equibounded as well and  taking into account the Leibniz rule we see that $\eta f_n\circ u\in W^{1,2}(\X,\sfd_\X,\mm)$ with equibounded norm. Since we assumed such space to be reflexive, up to pass to a non-relabeled subsequence we can assume that $(\eta f_n\circ u)_n$ has a $W^{1,2}$-weak limit and it is then clear that such limit is $\eta g$. Thus we have that $(\d(\eta f_n\circ u) )_n$ converges to $\d (\eta g)$ weakly in $L^2(T^*\X)$ and, by the choices of $v,\eta$, this implies that $(\d( f_n\circ u)(v))_n$ weakly converges to $\d g(v)$ in $L^2(\X)$. Now notice that 
\[
\d(f_n\circ u)(v)=[u^*\d_{\mm_\Y}f_n](\hat \d u(v))\leq|\d_{\mm_\Y}f_n|\circ u|\hat \d u(v)|\leq \lip_a(f_n)\circ u|\hat \d u(v)|.
\]
This, \eqref{eq:perdopo} and the choice of $(f_n)$ give that the rightmost side of the estimate above converges to the right hand side of \eqref{eq:toprove} in $L^2(\X)$. This concludes the argument.

\noindent{\bf  $W^{1,2}(\Y,\sfd_\Y,\mu)$ is reflexive.} According to \cite[Proposition 7.6]{ACM14} and its proof, in this case for any $f\in W^{1,2}(\Y,\sfd_\Y,\mu)$ we can find  $(f_n)\subset \LIP_{bs}(\Y)\subset W^{1,2}(\Y,\sfd_\Y,\mm_\Y)$ converging to $f$ in $W^{1,2}(\Y,\sfd_\Y,\mu)$ and such that $\lip_a(f_n)\to |\d_\mu f|$ in $L^2(\mu)$. The definitions of $\d u,\hat \d u$ give
\[
{\rm ext}[u^*\d_\mu f_n](\d u(v))\stackrel{\eqref{eq:glip}}=\d(f_n\circ u)(v)=[u^*\d_{\mm_\Y}f_n](\hat \d u(v))\leq|\hat \d u(v)|\,|\d_{\mm_\Y}f_n|\circ u\leq |\hat \d u(v)|\lip_a(f_n)\circ u
\]
and since the construction also ensures that $[u^*\d_\mu f_n]\to [u^*\d_\mu f]$ as $n\to\infty$, by passing to the limit we get that 
\[
{\rm ext}([u^*\d_\mu f])(\d u(v))\leq  |\hat \d u(v)|\,|\d_\mu f|\circ u= |\hat \d u(v)|\,|{\rm ext}([u^*\d_\mu f])|,\quad\mm-a.e..
\]
By the arbitrariness of $f\in W^{1,2}(\Y,\sfd_\Y,\mu)$, this is sufficient to conclude the proof.
\end{proof}
\subsection{The case $\X=\R^d$ and $u$ Lipschitz}\label{se:kir} In this section we assume that our source space $\X$ is $(\R^d,\sfd_{\rm Eucl},\mathcal L^d)$ and that the map $u\in \Sob^2(\R^d,\Y)$ is also Lipschitz. In this case Kirchheim proved in \cite{Kir94} that for $\mathcal L^d$-a.e.\ $x\in\R^d$ there is a seminorm ${\rm md}(u,x)$ on $\R^d$ such that:
\[
\text{for $\mathcal L^d$-a.e.\ $x$ we have}\qquad{\rm md}(u,x)(v)=\lim_{t\downarrow 0}\frac{\sfd_\Y\big(u(x+tv),u(x)\big)}{t}\qquad\text{ for every }v\in\R^d,
\]
where it is part of the claim the fact that the limit in the right hand side exists for $\mathcal L^d$-a.e.\ $x$.

We now show that such concept is fully compatible with  the notion of differential we introduced:
\begin{theorem}\label{kirchheim} Let $u:\R^d\to \Y$ be a Lipschitz map that is also in ${\rm S^2}(\R^d,\Y)$ and  $v\in\R^d\sim T\R^d$. Denote by $\bar v\in L^0(T\R^d)$ the vector field constantly equal to $v$. Then
\begin{equation}\label{eq:equality_with_md}
\big|\d u(\bar v)\big|(x)={\rm md}(u,x)(v)\quad\text{ for }\mathcal L^d\text{-a.e.\ }x\in\R^d.
\end{equation}
\end{theorem}
\begin{proof}\ \\
$\boxed{\geq}$ Let $(y_n)_n$ be countable and dense in $u(\R^d)\subset \Y$ and, for any $n\in\N$, put $f_n(\cdot):=\sfd_\Y(\cdot,y_n)$. From the compatibility of the abstract differential with the classical distributional notion in the case $\X=\R^d$ (see \cite[Remark 2.2.4]{Gigli17}) and Rademacher's theorem we see that
\begin{equation}
\label{eq:weakcomp}
\d(f_n\circ u)(\bar v)=\lim_{h\to 0}\frac{f_n\circ u(\cdot+hv)-f_n\circ u(\cdot)}{h}\qquad\mathcal L^d-a.e..
\end{equation}
For $x\in \R^d$ let $\gamma^x:[0,1]\to \Y$ be the Lipschitz curve defined by $\gamma^x_t:=u(x+tv)$ and  put $g^x_{n,t}:=f_n\circ\gamma_t^x$. By \cite[Theorem 1.1.2]{AmbrosioGigliSavare08} and its proof we know that for  the metric speed $|\dot\gamma^x_t|$ it holds $|\dot\gamma^x_t|=\sup_n\partial_t g^x_{n,t}$ for every $x\in\R^d$ and a.e.\ $t$, so that taking \eqref{eq:weakcomp} into account we obtain
\[
{\rm md}(u,x+tv)(v)=|\dot\gamma^x_t|=\sup_n\partial_t g^x_{n,t}=\sup_n\d(f_n\circ u)(\bar v)(x+tv)\qquad\mathcal L^d-a.e.\ x,\ a.e.\ t.
\] 
Hence Fubini's theorem yields
\[
{\rm md}(u,\cdot)(v)=\sup_n\d(f_n\circ u)(\bar v)\stackrel{\eqref{eq:glip}}=\sup_n{\rm ext}([u^*\d_\mu f_n])(\d u(\bar v))\leq |\d u(\bar v)|\qquad\mathcal L^d-a.e.,
\]
having used the trivial bound $|\d_\mu f_n|\leq 1$ $\mu$-a.e.\ in the last step.

\noindent$\boxed{\leq}$ Let  $f\in   \Sob^{2}(\Y,\sfd_\Y,\mu)$ be arbitrary and $g\in \Sob^2(\X)$ as in Proposition \ref{prop:compo}. We will show that 
\begin{equation}
\label{eq:toprove2}
\d g(v)\leq |\d_\mu f|\circ u\,{\rm md}(u,\cdot)(v) \qquad\mathcal L^d-a.e.,
\end{equation}
which is sufficient to conclude. The bound $\geq$ in \eqref{eq:equality_with_md} that we already proved and the same arguments used in studying \eqref{eq:toprove} show that the right hand side of \eqref{eq:toprove2} is well defined $\mathcal L^d$-a.e.\ and that 
\begin{equation}
\label{eq:contl2}
L^2(\Y,\mu)\ni G\quad\mapsto\quad G\circ u\,{\rm md}(u,\cdot)(v)\in L^2(\R^d)\qquad\text{ is linear and continuous}.
\end{equation}
Now let $(f_n)\subset\LIP_{bd}(\Y)$ be as in Proposition \ref{prop:compo} and notice that for every $n\in\N$ the identity  \eqref{eq:weakcomp} yields, for $\mathcal L^d$-a.e.\ $x$:
\[
|\d (f_n\circ u)(\bar v)|(x)\leq \lip_a(f_n)(u(x))\lim_{h\to 0}\frac{\sfd_\Y\big(u(x+hv),u(x)\big)}{|h|}=\big(\lip_a(f_n)\circ u\big)(x)\,{\rm md}(u,x)(v).
\]
By \eqref{eq:contl2} and the choice of $(f_n)$ we see that the rightmost side of the above converges to the right hand side of \eqref{eq:toprove2} in $L^2(\R^d)$ and again following the arguments in the first part of the proof of Proposition \ref{thm:equal_diff} (applicable, as $W^{1,2}(\R^d)$ is certainly reflexive) we see that $\big(\d (f_n\circ u)(\bar v)\big)_n$ converges to $\d g(v)$ weakly in $L^2(\R^d)$. Hence \eqref{eq:toprove2} is obtained.
\end{proof}

\section{Differential of locally Sobolev maps between metric spaces}\label{s:local_Sob}
\subsection{Inverse limits of modules} Here we briefly discuss properties of inverse limits in the category of $L^0(\mm)$-normed modules, where morphisms are $L^0(\mm)$-linear contractions, i.e.\  maps  $T:\mathscr M\to\mathscr N$ such that $|T(v)|\leq |v|$ $\mm$-a.e.. 
We start with:
\begin{proposition}\label{prop:invlim}
Let $(\{\mathscr M_i\}_{i\in I},\{P_j^i\}_{i\leq j\in I})$ be an inverse system of $L^0(\mm)$-normed modules. Then there exists the inverse limit $(\mathscr M,\{P^i\}_{i\in I})$. For every family $I\ni i\mapsto v^i\in \mathscr M_i$ such that
\begin{equation}
\label{eq:compvi}
P_j^i(v^j)=v^i\qquad\text{and}\qquad \esssup_{i\in I} |v^i|\in L^0(\mm)
\end{equation}
there is a unique $v\in\mathscr M$ such that $v^i=P^i(v)$ for every $i\in I$ and it satisfies $|v|= \esssup_i |v^i|$.
\end{proposition}
\begin{proof}
The system $(\{\mathscr M_i\}_{i\in I},\{P_j^i\}_{i\leq j\in I})$  is also an inverse system in the category of algebraic modules over the ring $L^0(\mm)$ in the sense of \cite[Chapter III.\S10]{lang84}. Hence according to \cite[Chapter III, Theorem 10.2]{lang84} and its proof there exists the algebraic inverse limit $(\mathscr M_{\rm Alg},P_{\rm Alg}^i)$ and for every family $i\mapsto v^i\in\mathscr M_i$ there is a unique $v\in \mathscr M_{\rm Alg}$ such that $P_{\rm Alg}^i(v)=v^i$ for every $i\in I$. Now define $|v|$ for any
$v\in\mathscr M_{\rm Alg}$ as
\begin{equation}
\label{eq:defnormlim}
|v|:=\esssup_{i\in I}|P_{\rm Alg}^i(v)|
\end{equation}
so that $|v|:\X\to[0,+\infty]$ is the equivalence class of a Borel map up to $\mm$-a.e.\ equality, and put
\[
\mathscr M:=\big\{v\in\mathscr M_{\rm Alg}\, :\, |v|\in L^0(\mm)\big \}=\big\{v\in \mathscr M_{\rm Alg}\, :\, |v|<+\infty\ \mm-a.e.\big\},\qquad\qquad P^i:=P_{\rm Alg}^i\restr{\mathscr M}.
\]
We claim that $(\mathscr M,P^i)$ is the desired inverse limit. Start by noticing that \eqref{eq:defnormlim} ensures that $|P^i(v)|\leq |v|$ $\mm$-a.e.,\ i.e.\ the $P^i$'s are contractions, as required. Let us now check that  $\mathscr M$ is a $L^0(\mm)$-normed module: the only non-trivial thing to verify is that it is complete, i.e.\ that  if $(v^n)$ is Cauchy in $\mathscr M$, then it has a limit. Since the $P^i$'s are contractions, we see that $n\mapsto P^i(v_n)$ is Cauchy in $\mathscr M_i$ and thus has a limit $v^i$ for every $i\in I$. Passing to the limit in the identity $P^{i}(v^n)=P_{j}^{i}(P^{j}(v^n))$ valid for every $i\leq j$  and using the continuity of $P_j^i$ we deduce that $v^i=P_j^i(v^j)$, i.e.\ there is $v=(v^i)_{i\in I}\in \mathscr M_{\rm Alg}$. Since $(v_n)$ is Cauchy and, trivially, the pointwise norm in $\mathscr M$ satisfies the triangle inequality, we see that $(|v_n|)$ has a limit $f$ in $L^0(\mm)$. Then from the bound $|v^i|=\lim_n|P^i(v_n)|\leq \lim_n|v_n|=:f$ valid for every $i\in I$ we deduce $|v|\leq f$ and thus $v\in \mathscr M$. Similarly, from $|v^i-P^i(v_n)|=\lim_m |P^i(v_m)-P^i(v_n)|\leq \lim_m|v_m-v_n|$ we deduce $|v-v_n|\leq \lim_m|v_m-v_n|$ and passing to the $L^0(\mm)$-limit in $n$ and using that $(v_n)$ is $\mathscr M$-Cauchy we conclude that $v_n\to v$ in $\mathscr M$, thus proving completeness.

Now the fact that for $v^i$'s as in \eqref{eq:compvi} there is a unique $v\in\mathscr M$ projecting on them is a trivial consequence of the construction and from this fact the universality  
property of $(\mathscr M,P^i)$  follows.
\end{proof}
It is now easy to check that there exists the inverse limit of a compatible family of maps:
\begin{proposition}\label{prop:limitmap}
Let   $(\{\mathscr M^i\}_{i\in I},\{P^i_j\}_{i\leq j\in I})$ and $(\{\mathscr N^i\}_{i\in I},\{Q^i_j\}_{i\leq j\in I})$ be two inverse systems of $L^0(\mm)$-normed modules and $(\mathscr M,P^i),(\mathscr N, Q^i)$ their inverse limits. Also, for every $i\in I$ let $T^i:\mathscr M^i\to \mathscr N^i$ be $L^0(\mm)$-linear and continuous and such that 
\begin{equation}
\label{eq:comppq}
T^i\circ P_j^i=Q_j^i\circ T^j\qquad\forall i\leq j\in I
\end{equation}
and so that for some $\ell\in L^0(\mm)$ we have
\begin{equation}
\label{eq:unifbound}
|T^i(v^i)|\leq \ell |v^i|\quad\mm-a.e.\qquad \forall i\in I,\ v^i\in\mathscr M^i.
\end{equation}
Then there exists a unique $L^0(\mm)$-linear and continuous map $T:\mathscr M\to \mathscr N$ such that $Q^i\circ T=T^i\circ P^i$ for every $i\in I$ and it satisfies $|T(v)|\leq \ell |v|$ $\mm$-a.e.\ for every $v\in\mathscr M$.
\end{proposition}
\begin{proof}
Let $v\in\mathscr M$, put $w^i:=T^i(P^i(v))\in \mathscr N^i$ and notice that \eqref{eq:comppq} yields $Q_j^i(w^j)=w^i$ and \eqref{eq:unifbound} that $|w^i|\leq \ell |v|$ $\mm$-a.e.\ for every $i\leq j\in I$. Thus Proposition \ref{prop:invlim} above ensures that there is a unique $T(v)\in\mathscr N$ such that $Q^i(T(v))=w^i$ for every $i\in I$ and it satisfies $|T(v)|\leq \ell |v|$ $\mm$-a.e.. Since the assignment $v\mapsto T(v)$ is trivially $L^0(\mm)$-linear, the proof is completed.
\end{proof}

\subsection{Locally Sobolev maps and their differential}

In this section we come back to the case of general $(\X,\sfd_\X,\mm)$, $(\Y,\sfd_\Y)$ as in Section \ref{s:diff_Sob_map} and study the case of $u\in \Sob^2_{\rm loc}(\X,\Y)$, this being the collection of functions $u$ such that  every $x\in\X$ has a neighbourhood $U_x$ such that $u$ coincides with some $u_x\in\Sob^2(\X,\Y)$ $\mm$-a.e.\ in $U_x$. Then for $u\in  \Sob^2_{\rm loc}(\X,\Y)$ the locality of the differential ensures that the formula
\[
|D u|:=|D u_x|\quad\mm-a.e.\ on\ U_x\qquad\forall x\in\X
\]
gives a well-defined function $|Du|\in L^2_{\rm loc}(\X)$. Here $L^2_{\rm loc}(\X)$  denotes the space of locally square-integrable functions on $\X$.

For this kind of $u$ the measure $u_*(|D u|^2\mm)$ is in general not $\sigma$-finite any longer. Hence, to define the differential $\d u$ we need to suitably adapt the definition previously given. This is the scope of the current section.

\bigskip

Fix $u\in \Sob^2_{\rm loc}(\X,\Y)$. By $\F(u)$ we denote the collection of open sets $\Omega\subset\X$ such that $\int_\Omega |Du|^2\,\d\mm<\infty$. Since  $u\in{\rm S}_{\rm loc}^2(\X,\Y)$ we see that $\F(u)$ is a cover of $\X$. We shall now build two inverse limits of $L^0(\mm)$-normed modules indexed over $\F(u)$, directed by inclusion. For the first define, for $\Omega\in \F(u)$, the measure $\mu_\Omega$ on $\Y$ as
\[
\mu_\Omega:=u_\ast(|Du|^2\mm\restr\Omega).
\]
Thus $\mu_\Omega$ is Radon and we can consider the cotangent module $L^0_{\mu_\Omega}(T^*\Y)$ of $(\Y,\sfd_\Y,\mu_\Omega)$ and its pullback $u^*L^0_{\mu_\Omega}(T^*\Y)$ which is a $L^0(\mm\restr{\Omega\cap\{|Du|>0\}})$-normed module. Then  put   $u^*L^0_{\Omega}(T^*\Y):={\rm Ext}(u^*L^0_{\mu_\Omega}(T^*\Y))$, which is $L^0(\mm)$-normed. Observe that for $\Omega'\subset\Omega\in\F(u)$ we have $\mu_{\Omega'}\leq\mu_\Omega$ and thus Lemma \ref{le:changemeas} provides a canonical `projection' map $P_\Omega^{\Omega'}:L^0_{\mu_\Omega}(T^*\Y)\to {\rm Ext}(L^0_{\mu_{\Omega'}}(T^*\Y))$. Then we can consider the (extended)  pullback  map $u^*P_\Omega^{\Omega'}:u^*L^0_{\Omega}(T^*\Y)\to u^*L^0_{\Omega'}(T^*\Y)$ and notice that  since $P_{\Omega_2}^{\Omega_1}\circ P_{\Omega_3}^{\Omega_2}=P_{\Omega_3}^{\Omega_1}$ for every $\Omega_3\subset\Omega_2\subset\Omega_1\in \F(u)$, the functoriality of the pullback grants that $(u^*L^0_{\Omega}(T^*\Y),u^*P_{\Omega}^{\Omega'})$ is an inverse system of $L^0(\mm)$-normed modules. We then call $(u^*L^0_u(T^*\Y), P^\Omega)$ its inverse limit (recall Proposition \ref{prop:invlim}).
\begin{remark}{\rm
For every $f:\Y\to\R$ Lipschitz with bounded support we have $f\in\Sob^2(\Y,\sfd_\Y,\mu)$ and $|\d_\mu f|\leq \Lip(f)$ $\mu$-a.e.\ for every finite Radon measure $\mu$. Hence there is an element $\omega\in u^*L^0_u(T^*\Y)$ such that $P^\Omega(\omega)={\rm ext}([u^*\d_{\mu_\Omega} f])$ for every $\Omega\in \F(u)$.  
}\fr\end{remark}

For the second consider, given $\Omega\subset\X$ open, the $L^0(\mm\restr\Omega)$-normed module $L^0_{\mm\restr\Omega}(T^*\X)$ and its extension $L^0_\Omega(T^*\X):={\rm Ext}(L^0_{\mm\restr\Omega}(T^*\X))$ which is $L^0(\mm)$-normed. Since trivially for $\Omega'\subset\Omega$ we have $\mm\restr{\Omega'}\leq\mm\restr\Omega$, Lemma \ref{le:changemeas} grants the existence of canonical (extended) `projection' maps $Q_\Omega^{\Omega'}:L^0_\Omega(T^*\X)\to L^0_{\Omega'}(T^*\X)$ and by construction it is clear that $\big(\big\{L^0_{\Omega}(T^*\X)\big\}_{\Omega\in\F(u)},\{Q_{\Omega}^{\Omega'}\}_{\Omega'\subset\Omega}\big)$ is an inverse system of $L^0(\mm)$-normed modules. We then have the following non-obvious result:
\begin{lemma}\label{le:limitcot}
The inverse limit of $\big(\big\{L^0_{\Omega}(T^*\X)\big\}_{\Omega\in\F(u)},\{Q_{\Omega}^{\Omega'}\}_{\Omega'\subset\Omega}\big)$  is $\big(L^0(T^*\X),\{Q_{\X}^{\Omega}\}_{\Omega'\subset\Omega}\big)$. 
\end{lemma}
\begin{proof}
The fact that $Q_{\Omega}^{\Omega'}\circ Q_{\X}^{\Omega}=Q_{\X}^{\Omega'}$ for $\Omega'\subset\Omega\in\F(u)$ is a direct consequence of the definition of the $Q$'s. 
For universality, we recall  \cite[Theorem 4.19]{AmbrosioGigliSavare11-2} and its proof (in particular: the assumption $\mm(\X)=1$ plays no role) to get  that $|Q_\Omega^{\Omega'}(\d_{\mm\restr\Omega}f)|=|\d_{\mm\restr{\Omega'}}f|$ $\mm$-a.e.\ on $\Omega'$ and that if $f\in \Sob^2(\X,\sfd_\X,\mm\restr{\Omega'})$ has support at positive distance from $\X\setminus \Omega'$, then $f\in \Sob^2(\X,\sfd_\X,\mm\restr\Omega)$ as well. 
It easily follows that $Q_\Omega^{\Omega'}:L^0_\Omega(T^*\X)\to L^0_{\Omega'}(T^*\X)$ has a unique norm-preserving right inverse, call it $P_{\Omega'}^\Omega$. Then if $\F(u)\ni \Omega\mapsto \omega^\Omega\in L^0_{\Omega}(T^*\X)$ satisfies $Q_\Omega^{\Omega'}(\omega^\Omega)=\omega^{\Omega'}$ for every  $\Omega'\subset\Omega\in\F(u)$, it is clear that there is a unique $\omega\in L^0(T^*\X)$ such that $\nchi_\Omega\,\omega=P_\Omega^\X(\omega^\Omega)$ for every $\Omega\in\F(u)$ and this is sufficient to conclude.
\end{proof}
Let  $\Omega\in\F(u)$ and define $S_\Omega:\{\d_{\mu_\Omega} f:f\in\Sob^2(\Y,\sfd_\Y,\mu_\Omega)\}\to  L^0_\Omega(T^*\X)$ by putting
\[
S_\Omega(\d_{\mu_\Omega} f):={\rm ext}( \d_{\mm\restr\Omega} g),
\]
where $g$ is related to $f$ as in Proposition \ref{prop:compo}, here applied to the space $(\X,\sfd_\X,\mm\restr\Omega)$. In particular the bound \eqref{eq:chain} gives
\begin{equation}
\label{eq:fortom}
|S_\Omega(\d_{\mu_\Omega} f)|\leq \nchi_\Omega\big(|\d_{\mu_\Omega}f|\circ u|D u|\big)
\end{equation}
which is easily seen to ensure that $S_\Omega$ is well posed (i.e.\ the value of $S_\Omega$ depends only on $\d_{\mu_\Omega} f$ and not on $f$). Thus by the universality property of the pullback we see that there exists a unique $L^0(\mm)$-linear and continuous map $T_\Omega:u^*L^0_{\Omega}(T^*\Y)\to L^0_\Omega(T^*\X)$ such that
\[
T_\Omega({\rm ext}([u^*\d_{\mu_\Omega}f]))=S_\Omega(\d_{\mu_\Omega}f)\qquad\forall f\in \Sob^2(\Y,\sfd_\Y,\mu_\Omega)
\]
and by \eqref{eq:fortom} such $T_\Omega$ satisfies
\begin{equation}
\label{eq:tonormt}
|T_\Omega(\omega)|\leq |Du||\omega|\quad\mm-a.e.\qquad\forall \omega\in u^*L^0_{\Omega}(T^*\Y).
\end{equation}
It is now only a matter of keeping track of the various definitions to check that  for every $\Omega'\subset\Omega\in\F(u)$ it holds
\begin{equation}
\label{eq:conj}
T_{\Omega'}(u^*P_\Omega^{\Omega'}(\omega))=Q_\Omega^{\Omega'}(T_\Omega(\omega))
\end{equation}
for every $\omega\in u^*L^0_{\Omega}(T^*\Y)$ of the form $\omega={\rm ext}([u^*\d_{\mm\restr\Omega}f])$ for some $f\in \Sob^2(\Y,\sfd_\Y,\mu_\Omega)$. Then by $L^0(\mm)$-linearity and continuity we see that \eqref{eq:conj} holds for every $\omega\in u^*L^0_{\Omega}(T^*\Y)$. In light of \eqref{eq:tonormt}, Proposition \ref{prop:limitmap}  and Lemma \ref{le:limitcot} we have that there is a unique $L^0(\mm)$-linear and continuous map $T: u^*L^0_u(T^*\Y)\to L^0(T^*\X)$ such that
\[
Q_\X^\Omega(T(\omega))=T_\Omega(P^\Omega(\omega))\qquad\forall \omega\in u^*L^0_u(T^*\Y),\ \Omega\in\F(u).
\]
We can now give the main definition of this section:
\begin{definition}\label{def:diffloc}
The differential $\d u: L^0(T\X)\to (u^*L^0_u(T^*\Y))^*$ is defined as the adjoint of $T$.
\end{definition}
Notice that by \eqref{eq:tonormt} it follows that $|T(\omega)|\leq |D u||\omega|$  for every $\omega\in u^*L^0_u(T^*\Y)$. Hence by duality we also get that $|\d u(v)|\leq |Du||v|$ $\mm$-a.e.\ for every $v\in L^0(T\X)$, i.e.\ $|\d u|\leq |Du|$ $\mm$-a.e.. Then arguing as in Proposition \ref{basic} we can prove that actually $|\d u|=|Du|$ $\mm$-a.e.. Analogously, natural variants of the properties stated in Sections \ref{se:r}, \ref{se:bd}, \ref{se:kir} hold for this more general notion of differential. We omit the details.

We conclude observing that if $u\in \Sob^2(\X,\Y)\subset \Sob^2_{\rm loc}(\X,\Y)$, then $\X\in\F(u)$, i.e.\ the directed family $\F(u)$ has a maximum. It is then clear that the differential $\d u$ in the sense of Definition \ref{def:diffloc} canonically coincides with the one given by Definition \ref{def:differential_Sobolev_maps}.

\def\cprime{$'$} \def\cprime{$'$}

\end{document}